\documentclass[smallextended,envcountsect]{svjour3}
\usepackage{latexsym,amsmath,amsfonts,amscd,amssymb,verbatim,array,mathdots}
\smartqed

\usepackage{chngcntr}
\counterwithin{equation}{section}

\DeclareMathOperator{\gph}{gph}

\DeclareMathOperator{\Sol}{Sol}
\DeclareMathOperator{\OP}{OP}

\DeclareMathOperator{\B}{\mathcal{B}}
\DeclareMathOperator{\Bo}{\mathbb{B}}

\DeclareMathOperator{\E}{\mathcal{E}}

\DeclareMathOperator{\Hd}{\mathcal{H}}
\DeclareMathOperator{\KKT}{KKT}

\DeclareMathOperator{\Ri}{\mathcal{R}}
\DeclareMathOperator{\Oo}{\mathcal{O}}
\DeclareMathOperator{\U}{\mathcal{U}}
\DeclareMathOperator{\Sa}{\mathbb{S}}
\DeclareMathOperator{\rank}{rank}
\DeclareMathOperator{\R}{\mathbb{R}}

\begin{document}
\title{On the solution existence and stability of polynomial optimization problems}
\author{Vu Trung Hieu}

\institute{Vu Trung Hieu \at Sorbonne Universit\'e, \textsc{CNRS},
\textsc{LIP6}, F-75005, Paris, France
	\at Division of Mathematics, Phuong Dong University, 171 Trung Kinh Street,
	Hanoi, Vietnam\\
	E-mail: trung-hieu.vu@lip6.fr}

\date{Received: date / Accepted: date}

\maketitle

\begin{abstract}
In this paper, we introduce and investigate a new regularity condition in the asymptotic sense for optimization problems whose objective functions are polynomial. The normalization argument in asymptotic analysis enables us to study the existence as well as the stability of solutions of these problems. We prove a Frank-Wolfe type theorem for regular optimization problems and an Eaves type theorem for non-regular pseudoconvex optimization problems. Moreover, under the regularity condition, we show results on the stability such as upper semicontinuity and  local upper-H\"{o}lder stability of the solution map of polynomial optimization problems. At the end of the paper, we discuss the genericity of the regularity condition.
\end{abstract}



\section{Introduction} We consider the following optimization problem
\begin{equation*}\label{f}
{\rm minimize} \ \ f(x)\ \ {\rm subject \ to } \ \ x\in K,
\end{equation*}
where $K$ is a nonempty, closed subset of $\R^n$ and $f: \R^n\to\R$ is a polynomial in $n$ variables of degree $d\geq 2$.
The problem and its solution set are denoted by $\OP(K,f)$ and $\Sol(K,f)$ respectively. Let $f_d$ be the homogeneous component of degree $d$ of $f$, and let $K_{\infty}$ be the asymptotic cone of $K$ that will be introduced in Section \ref{sec:pre}. We say that $\OP(K,f)$ is regular if the solution set of the asymptotic problem $\OP(K_{\infty},f_d)$ is bounded, and the problem is non-regular otherwise. The regularity condition has appeared in studies about the solution existence and stability in quadratic programming (see, e.g., \cite{LTY05,TamNghi18} and the references therein).

Asymptotic cones and functions play an important role in optimization and
variational inequalities \cite{AuTe2003}. The normalization argument in
asymptotic analysis enables us to study the existence and stability of solutions not only for quadratic programming, linear complementarity problems, and affine
variational inequalities (see, e.g., \cite{LTY05,Cottle}), but also for
polynomial complementarity problems and polynomial variational inequalities
that have unbounded constraint sets (see, e.g., \cite{Gowda,Hieu}). In this paper,
the normalization argument is used as the main technique to investigate the
existence as well as the stability of solutions to polynomial optimization problems.

In 1956, Frank and Wolfe \cite{FW56} proved that if $K$ is polyhedral and $f$
is quadratic and bounded from below on $K$, then $\Sol(K,f)$ is nonempty.
Several versions of the Frank-Wolfe theorem for quadratic, cubic, and
polynomial optimization problems have been shown in
\cite{LTY05,TamNghi18,LouZ1999,BeKla2002,Obuchowska2006,DHP2014,Klatte2018}.
Belousov and  Klatte \cite{BeKla2002}, and Obuchowska
\cite{Obuchowska2006} have proved Frank-Wolfe type theorems for convex and quasiconvex polynomial optimization problems. Recently, by
using a technique from semi-algebraic geometry, Dinh, Ha
and Pham \cite{DHP2014} have shown a Frank-Wolfe type theorem for
nondegenerate problems. The present
paper gives another Frank-Wolfe type theorem, which says that if
$\OP(K,f)$ is regular and $f$ is bounded from below on $K$, then the problem
has a solution. Besides, the Eaves theorem \cite{Eaves71} provides us with another
criterion for the existence of solutions to quadratic optimization problems.
Extensions of this theorem for quadratically constrained quadratic problems
have been investigated in  \cite{LTY05,TamNghi18,KTY2012,NNS2020}. This paper
introduces an Eaves type theorem for non-regular pseudoconvex optimization
problems, where the constraint sets are convex.

Under the assumption that the constraint set $K$ is compact and semi-algebraic,
some stability and genericity results for polynomial optimization problems have
been shown by  Lee and Pham \cite{LP16}. If $K$ is compact, then its asymptotic
cone is trivial,
i.e., $K_{\infty}=\{0\}$; Hence that $\OP(K,f)$ satisfies
the regularity condition obviously. In the present paper, $K$ may be unbounded.
Under the regularity condition, we prove several local properties of the
solution map of polynomial optimization problems such as local boundedness and
upper semicontinuity. Furthermore, based on an error bound for a polynomial
system
in \cite{LMP14}, we prove the local upper-H\"{o}lder stability of the
solution map.

We denote by
$\R_d[x]$ the space of all polynomials of degree at most $d$ and by $\Ri_{d}$ the set of all polynomials $g$ of degree
$d$ such that $\OP(K,g)$ is regular. The set $\Ri_{d}$ is an open cone in
$\R_d[x]$. At the end of
this work, $K$ is defined by convex polynomials, we prove that $\Ri_{d}$ is generic
in $\R_d[x]$.

The organization of the paper is as follows.
Section \ref{sec:pre} gives a brief introduction to
asymptotic cones, polynomials, and the regularity condition.
Section~\ref{sec:exist} proves two criteria of the solution existence. Section
\ref{sec:stab} investigates properties of the solution map. The last section
discusses the genericity of the regularity condition.

\section{Preliminaries}\label{sec:pre}

Recall that the \textit{asymptotic cone} \cite{AuTe2003} of a nonempty closed subset $S$ in $\R^n$ is defined and denoted by
$$
S_{\infty}=\Big\{ v\in\R^n:\exists t_k\to +\infty, \exists x_k\in S \text{ with } \lim_{k\to\infty} \frac{x_k}{t_k}=v\Big\}.
$$
Clearly, the cone $S_{\infty}$ is closed and contains $0$. The set $S$ is bounded if and only if $S_{\infty}$ is trivial. Furthermore, if $S$ is convex then $S_{\infty}$ is a closed convex cone and $S_{\infty}=0^+S$, where $0^+S$ is the \textit{recession} cone of $S$, that consists of all vectors $v \in\R^n$ such that $x+tv\in S$ for any $x\in S$ and $t\geq 0$. Thus, one has $S=S+S_{\infty}$ when $S$ is convex.

Let $d\geq 2$ be given. The dimension of the space $\R_d[x]$ is finite; its dimension is denoted by $\rho$. Let  $X(x)$ be the vector consisting of $\rho$ monomials of degree at most $d$ which is listed by lexicographic ordering
\begin{equation*}\label{monomials}
X(x):= (1,x_1,x_2,\dots,x_n, x_1^2, x_1x_2,\dots,x_1x_n,\dots,x_1^d,x_1^{d-1}x_2,\dots,x_n^d)^T.
\end{equation*}
For every $g\in\R_d[x]$, there exists a unique vector $a=(a_1,\dots,a_{\rho})\in\R^{\rho}$
such that $g(x)= a^TX(x)$. We denote by $\|g\|$ the $\ell_2$--norm of the polynomial $g$, namely
$$\|g\|:=\|a\|=\sqrt{a_1^2+\dots+a_{\rho}^2}.$$
The Cauchy--Schwarz inequality yields
$|g(x)|\leq \|X(x)\|\|g\|.$
Furthermore, if $\{g^k\}$ is a convergent sequence in $\R_d[x]$ with $g^k\to g$, then $g^k_d\to g_d$.

Throughout the paper, we assume that the constraint set $K\subset\R^n$ is
nonempty and closed, and the objective function $f: \R^n\to\R$ is a polynomial
of degree $d\geq 2$.

We say that $\OP(K,f)$ is a \textit{polynomial optimization problem} if $K$ is given by polynomials. With the given set $K$ and the given integer $d$, the solution map of polynomial optimization problems $\OP(K,g)$, where $g\in \R_d[x]$, is defined by
\begin{equation*}\label{Sol}
	\Sol_K(\cdot):\R_d[x]\rightrightarrows \R^n, \ \ g\mapsto \Sol(K,g).
\end{equation*}

Assume that $g\in \R_d[x]$ with $\deg g=d$ and $g=g_d+\dots+g_{1}+g_0$, where $g_{l}$ is a homogeneous polynomial of degree $l$, i.e., $g_{l}(tx)=t^{l}g_{l}(x)$ for all $t\geq 0$ and $x\in\R^n$, $l\in[d]:=\{1,\dots,d\}$, and $g_0\in\R$. Then, $g_d$ is the \textit{leading term} (or the recession polynomial) of the polynomial $g$ (of degree $d$). Clearly, one has
$$g_d(x)=\lim_{\lambda\to+\infty}\frac{g(\lambda x)}{\lambda^d}, \ \forall x\in\R^n.$$

For the pair $(K,f)$, the asymptotic pair $(K_{\infty},f_d)$ is unique. The
asymptotic optimization problem $\OP(K_{\infty},f_d)$ plays a vital role in the
investigation of behavior of $\OP(K,f)$ at infinity. The following remarks
point out (without proof) the basic properties of the asymptotic problem.

\begin{remark}\label{sol_cone} Since $f_d$ is a homogeneous polynomial and
$K_{\infty}$ is a closed cone, the asymptotic optimization problem
$\OP(K_{\infty},f_d)$ has a solution if and only if $f_d$ is non-negative on
$K_{\infty}$.
\end{remark}

\begin{remark}\label{sol_cone0}
Assume that $\Sol(K_{\infty},f_d)$ is nonempty. Then,  this set is a closed
cone
with $0\in \Sol(K_{\infty},f_d)$. In addition, $\Sol(K_{\infty},f_d)$ coincides
with the zero set of $f_d$ in $K_{\infty}$, i.e.,
$$\Sol(K_{\infty},f_d)=\{x\in K_{\infty}: f_d(x)=0\}.$$
\end{remark}

Now, we introduce the regularity notion concerning the boundedness of the
solution set of $\OP(K_{\infty},f_d)$.

\begin{definition}\label{regular} The problem $\OP(K,f)$ is said to be \textit{regular} if $\Sol(K_{\infty},f_d)$ is bounded and \textit{non-regular} otherwise.
\end{definition}

Denote by $\E_{d}$ (resp., $\Oo_{d}$, $\U_{d}$) the set of all polynomials $g$
of degree $d$ such that $\Sol(K_{\infty},f_d)$ is the empty set (resp., the
trivial cone, an unbounded cone). Clearly, $\Ri_{d}=\E_{d}\cup\Oo_{d}$, and one
has the following disjoint union: \begin{equation}\label{Pd}
\R_d[x]=\R_{d-1}[x]\cup \E_{d}\cup\Oo_{d}\cup\U_{d}.
\end{equation}

\begin{remark} The boundedness of $\Sol(K_{\infty},f_d)$ implies that of $\Sol(K,f)$. Indeed, assume to the contrary that $\Sol(K,f)$ is unbounded. There exists an unbounded sequence $\{x^k\}\subset\Sol(K,f)$. Without loss of generality, we can assume
	$x^k$ is nonzero for all $k$, $\|x^k\| \to +\infty$,  and
	$\|x^k\|^{-1}x^k\to \bar x$ for some $\bar x\in\R^n$ with $\|\bar x\|=1$.
	Note that $f(x^k)= f^*$, where $f^*\in \R$ is the minimum of $f$ over $K$,
	for all $k$. By dividing the last equation by $\|x^k\|^d$ and letting $k\to
	+\infty$, we obtain $f_d(\bar x)= 0$. It follows that $\bar x \in
	\Sol(K_{\infty},f_d)$. Since $\bar x \neq 0$, the cone
	$\Sol(K_{\infty},f_d)$  is unbounded, which contradicts our assumption.
	Thus, the claim is proved.
\end{remark}

\begin{remark} We observe that the set $\Ri_{d}$ is nonempty. If $K$ is bounded then
$\Ri_{d}$ coincides with the set of all $g\in \R_d[x]$ such that $\deg g=d$.
Hence, we can suppose that $K$ is unbounded. Clearly, the cone $K_{\infty}$
also is unbounded. Let $\bar x \in K_{\infty}$ be nonzero. There exists
$l\in[n]$ such that $\bar x_l\neq 0$. Let us define a homogeneous polynomial of
degree $d$ as $f(x):=-(\bar x_l x_l)^d$. For any $t>0$, one has $t\bar x\in
K_{\infty}$, and $f(t\bar x)=-(\bar x_l^{2})^dt^d \to -\infty$ as $t\to
+\infty$. Then, $f$ is not bounded from below on $K_{\infty}$. This yields
$\Sol(K_{\infty},f)=\emptyset$ and $f\in \Ri_{d}$.
\end{remark}

\begin{example}\label{ex1}
	Consider the case that $n=1$, $K=\R$ and $d=2$. One has
	$\R_2[x]=\{a_2x^2+a_1x+a_0:(a_2,a_1,a_0)\in\R^3\}.$
	Since $K_{\infty}=\R$, an easy computation shows that
	 $\E_{2}=\{a_2x^2+a_1x+a_0:a_2< 0,a_1\in\R,a_0\in\R\}$, $\Oo_{2}=\{a_2x^2+a_1x+a_0:a_2> 0,a_1\in\R,a_0\in\R\}$, and $\Ri_{2}=\{a_2x^2+a_1x+a_0:a_2\neq 0,a_1\in\R,a_0\in\R\}$.
\end{example}

\section{Two criteria for the solution existence}\label{sec:exist}

We now introduce two criteria for the solution existence of $\OP(K,f)$. In the proofs, the normalization argument in asymptotic analysis plays a vital
role; meanwhile, the
semi-algebraicity of $K$ is not required.

\subsection{A Frank-Wolfe type theorem for regular problems}

The following theorem provides us a criterion for the solution existence of regular optimization problems.

\begin{theorem}[Frank-Wolfe type theorem]\label{thm:FW} If $\OP(K,f)$ is regular and $f$ is bounded from below on $K$, then its solution set is nonempty and compact.
\end{theorem}

\begin{proof}
Suppose that $f\in \Ri_{d}$, i.e. $f\in \E_{d}\cup\Oo_{d}$,  and there exists
$\gamma\in\R$ such that $\gamma\leq f(x)$, for all $x\in K$. For any given
$v\in K_{\infty}$, there are two sequences $\{t_k\}\subset\R_+$ and
$\{x^k\}\subset K$ such that $t_k\to +\infty$ and $t_k^{-1}x^k\to v$ as $k\to
+\infty$. For any $k$, one has	$ \gamma\leq f(x^k)$. Dividing both sides of
the last inequality by $t_k^d$ and letting $k\to +\infty$, we obtain $0\leq
f_d(v)$. Thus, $f_d$ is non-negative over $K_{\infty}$. It follows from
Remark~\ref{sol_cone} that $f$ does not belong to $\E_{d}$; hence, we conclude
that $f$ must be in $\Oo_{d}$.

Let $\bar x \in K$ be given, and
$M:=\{x\in K: f(x)\leq f(\bar x)\}$. It is easy to
check that $\Sol(M,f)=\Sol(K,f)$. Hence, we need only to prove that $\Sol(M,f)$
is nonempty and compact.

Clearly, $M$ is closed. We claim that $M$ is bounded. On the contrary, we
suppose that there exists an
unbounded sequence $\{x^k\}\subset M$ such that
$x^k$ is nonzero for all $k$, $\|x^k\| \to +\infty$, and $\|x^k\|^{-1}x^k\to v$ for some $v\in\R^n$ with $\|v\|=1$.
One has
\begin{equation}\label{ffg}
\gamma \leq f(x_k) \leq f(\bar x),
\end{equation}
for all $k$. Dividing the values in \eqref{ffg} by $\|x^k\|^d$ and letting
$k\to +\infty$, we get $f_d(v)= 0$. This yields $v\in \Sol(K_{\infty},f_d)$.
Because of $v\neq 0$, $\Sol(K_{\infty},f_d)$ is unbounded. This contradicts our
assumption and, thus, the claim is proved.

The compactness of $M$ and Bolzano-Weierstrass' Theorem allow us to conclude
that $\Sol(M,f)$
is nonempty and compact. \qed
\end{proof}

\begin{remark}\label{rk:O} From the proof of Theorem \ref{thm:FW}, we see that
if $f$ is bounded from below on $K$ then $\Sol(K_{\infty},f_d)$ is nonempty,
i.e. $f\in \Oo_{d}\cup \U_{d}$. Hence, if $f\in \E_{d}$ then $\OP(K,f)$ has no
solution.
\end{remark}

\begin{corollary}
Assume that $f=\alpha_1x_1^d+\dots+\alpha_nx_n^d+p$ where $d$ is even, $\alpha_{\ell}>0$ for all $\ell \in [n]$,  $p$ is a polynomial with $\deg p<d$. Then, $\Sol(K,f)$ is nonempty and compact.
\end{corollary}

\begin{proof}Clearly, $f_d=\alpha_1x_1^d+\dots+\alpha_nx_n^d$ is non-negative
over $\R^n$. It follows that $f_d$ is also non-negative over $K_{\infty}$. From
Remarks \ref{sol_cone} and \ref{sol_cone0}, it is clear that
$\Sol(K_{\infty},f_d)$ is nonempty and
$$\Sol(K_{\infty},f_d)=\{x\in K_{\infty}: \alpha_1x_1^d+\dots+\alpha_nx_n^d=0\}=\{0\}.$$
This means that $f\in \Oo_{d}$. Clearly, $f$ is bounded from below on $K$, and
the condition of Theorem~\ref{thm:FW} holds.
Therefore, $\Sol(K,f)$ is nonempty and compact. \qed
\end{proof}

The following example illustrates Theorem \ref{thm:FW}, in which the constraint set is neither convex nor semi-algebraic.

\begin{example}
Consider the optimization problem $\OP(K,f)$, where the polynomial $f$ is given
by $f(x_1,x_2)=x^3_2-x_1x_2$  and the constraint set $K$ is given by
$$K=\{(x_1,x_2)\in\R^2: x_1\geq 0, x_2-x_1 \geq 0,  e^{x_1}-x_2 \geq 0\}.$$
Since $f_3(x_1,x_2)=x_2^3$ and $K_{\infty}=\{(x_1,x_2)\in\R^2: x_1 \geq
0,x_2-x_1\geq 0\}$,
	one has $\Sol(K_{\infty},f_3)=\{(0,0)\}$. According to Theorem~\ref{thm:FW}, $\Sol(K,f)$ is nonempty and compact.
\end{example}

\subsection{An Eaves type theorem for non-regular problems}

In this subsection, we investigate the solution existence of non-regular
optimization problems, where the objective functions are pseudoconvex on the
constraint sets.

Assume that $U$ is an open subset of $\R^n$. One says the polynomial $f$ is
\textit{pseudoconvex} on $U$ if, for any $x,y\in U$ such that $\langle \nabla
f(x),y-x\rangle\geq 0$, here $\nabla f$ is the gradient of $f$, we have
$f(y)\geq f(x)$. Recall that $f$ is pseudoconvex on $U$ if and only if $\nabla
f$ is pseudomonotone on $U$ \cite[Theorem 3.1]{Kara}, i.e. if, for any $x,y\in
U$ such that $\langle \nabla f(x),y-x\rangle\geq 0$, we have $\langle \nabla
f(y),y-x\rangle\geq 0$.

\begin{lemma}\label{lm_Minty} Assume that $K$ is convex and $f$ is pseudoconvex on an open set $U$ containing $K$. If $x^0\in \Sol(K,f)$, then $\langle \nabla f(x),x-x^0\rangle\geq 0$ for all $x\in K$.
\end{lemma}
\begin{proof} Since $f$ is pseudoconvex on the set $U$, the gradient $\nabla f$ is pseudomonotone on $U$. Suppose that $x^0\in \Sol(K,f)$, one has $\langle \nabla f(x^0),x-x^0\rangle\geq 0$ for all $x\in K$ (see, e.g., \cite[Proposition 5.2]{ALM2014}). The pseudomonotonicity of the gradient implies that $\langle \nabla f(x),x-x^0\rangle\geq 0$ for all $x\in K$.	The lemma is proved.  \qed
\end{proof}

\begin{theorem}[Eaves type theorem]\label{thm:Eaves} Assume that $K$ is convex
and $f$ is pseudoconvex on an open set containing $K$. If $\OP(K,f)$ is
non-regular, then the following statements are equivalent:
	\begin{description}
		\item[\rm(a)] If $v\in\Sol(K_{\infty},f_d)\setminus\{0\}$, then there exists $x\in K$ such that $\left\langle \nabla f(x), v \right\rangle > 0$;
		\item[\rm(b)] $\Sol(K,f)$ is nonempty and compact.
	\end{description}
\end{theorem}

\begin{proof} Suppose that  $\OP(K,f)$ is
	non-regular. We prove $\rm(a) \Rightarrow \rm(b)$. Assume that $\rm(a)$
	holds.  For
	each $k\in\mathbb{N}$, we denote $$K_k=\{x\in\R^n:x\in K,\|x\|\leq k\}.$$
Clearly, $K_k$ is compact and convex. Without loss of generality, we can assume
that $K_k$
is nonempty. According to Bolzano-Weierstrass' Theorem, $\OP(K_k,f)$ has a
solution, denoted by $x^k$.

We assert that the sequence $\{x^k\}$ is bounded. Indeed, suppose on the contrary that $\{x^k\}$ is unbounded, $x^k\neq 0$ for all $k$, $\|x^k\| \to +\infty$, and $\|x^k\|^{-1}x^k\to v$, where $ v \in K_{\infty}$ and $\| v\|=1$.	For each $k$, one has
	\begin{equation}\label{f_Kk}
	f(x^k) \leq f(x), \ \forall x\in K_k.
	\end{equation}
Let $y\in K$ be given. For
$k$ large enough, $y\in K_{k}$ and $f(x^k) \leq f(y)$. By  dividing two sides
 of the last inequality by
$\|x^k\|^{d}$ and letting $k\to+\infty$, we obtain
	$f_{d}(v)\leq 0$. This leads to $ v\in\Sol(K_{\infty},f_{d})\setminus\{0\}$.
Furthermore, since $f$ is pseudoconvex on $K_k$, from Lemma
\ref{lm_Minty} we have
	\begin{equation}\label{nab_f}
	\langle \nabla f(y),y-x^k\rangle\geq 0.
	\end{equation}
Dividing both
	sides of the inequality in \eqref{nab_f} by $\|x^k\|$	and letting $k\to
	+\infty$, we obtain $\langle \nabla f(y),v\rangle\leq 0$. The conclusion
	holds for any $x\in K$, i.e., $\langle \nabla f(x),v\rangle\leq 0$ for all
	$x\in K$. This contradicts to our assumption.  Hence, $\{x^k\}$ is bounded.

We can assume that  $x^k\to \bar x$. From \eqref{f_Kk}, by the continuity
of $f$, it easy to check that  $\bar x$ solves $\OP(K,f)$, so $\Sol(K,f)$ is
nonempty.

	To prove the compactness of the solution set, we can repeat the previous argument by supposing that there is an unbounded solution sequence $\{x^k\}$, and can show that there exists $ v\in\Sol(K_{\infty},f_{d})\setminus\{0\}$ such that $\langle \nabla f(x),v\rangle\leq 0$ for all $x\in K$. This contradicts to $\rm(b)$.

	$\rm(b) \Rightarrow \rm(a)$ Since $K$ is convex, one has $K_{\infty}=0^+K$ and $K=K+K_{\infty}$. Suppose that $\Sol(K,f)$ is nonempty and compact, but $\rm(b)$ is wrong, i.e. there exists $v\in\Sol(K_{\infty},f_d)\setminus\{0\}$ such that $\langle \nabla f(x),v\rangle\leq 0$ for all $x\in K$. Let $\bar x$ be a solution of $\OP(K,f)$. For any $t\geq 0$, one has $\bar x+tv\in K$ and $\langle \nabla f(\bar x+tv),v\rangle\leq 0$. Thus, we have
	$$\langle \nabla f(\bar x+tv),\bar x-(\bar x+tv)\rangle\geq 0.$$
	The pseudoconvexity of $f$ yields  $f(\bar x)\geq f(x^0+tv)$. Hence,
	$x^0+tv$ belongs to $\Sol(K,f)$, for any $t\geq 0$. This shows that
	$\Sol(K,f)$ is unbounded which contradicts to our assumption. Thus $\rm
	(a)$ holds, and the proof is complete. \qed
\end{proof}

\begin{corollary}\label{cor:Eaves} Assume that $K$
is convex and $f$ is convex on an open convex set containing $K$. If $\OP(K,f)$
is non-regular, then the
following statements are equivalent:
	\begin{description}
		\item[\rm(a)] If $v\in\Sol(K_{\infty},f_d)\setminus\{0\}$, then there
		exists $x\in K$ such that $\left\langle \nabla f(x), v \right\rangle >
		0$;
		\item[\rm(b)] $\Sol(K,f)$ is nonempty and compact.
	\end{description}
\end{corollary}
\begin{proof}
	Since the convexity implies the pseudoconvexity, by applying Theorem \ref{thm:Eaves} for the convex polynomial $f$, we have the assertion. \qed
\end{proof}

The following example illustrates Corollary \ref{cor:Eaves}.

\begin{example}	Consider the polynomial optimization problem $\OP(K,f)$ with
the objective function $f(x_1,x_2)=\frac{1}{6}x_2^3+\frac{1}{2}x_1^2-x_1x_2$
and the constraint set $K=\{(x_1,x_2)\in\R^2: x_1x_2\geq 1, x_2\geq 2\}.$ The
gradient and the Hessian matrix of $f$, respectively, are given by
	$$\nabla f= \begin{bmatrix}
	x_1-x_2 \\
	-x_1+\frac{1}{2}x_2^2
	\end{bmatrix}, \ \ H=\begin{bmatrix}
	1& -1 \\
	-1 \ & \ x_2
	\end{bmatrix}.$$
	It is easy to check that $K$ is convex and $H$ is positive semidefinite on the open set $U=\{(x_1,x_2)\in\R^2: x_1x_2 > 0, x_2 > 1\}\supset K$; hence $f$ is convex on $K$. One has $K_{\infty}=\R^2_+$ and $f_3(x_1,x_2)=\frac{1}{6}x_2^3$. This yields $$\Sol(K_{\infty},f_3)=\{(x_1,x_2)\in\R^2: x_1\geq 0,x_2=0\}.$$
For every $v=(\alpha,0)$ in  $\Sol(K_{\infty},f_3)\setminus\{0\}$,  one has
$\alpha > 0$. By choosing the point $x=(3,2)$ in the constraint set, we get
$\left\langle \nabla f(x), v \right\rangle=\alpha > 0$. Finally, according to
Corollary
\ref{cor:Eaves}, the solution set of $\OP(K,f)$ is nonempty and compact.
\end{example}

\section{Stability of the solution map}\label{sec:stab}

We investigate the local
boundedness, the upper semicontinuity, and the local upper-H\"{o}lder stability
 of the solution map under the regularity condition.

\subsection{Upper semicontinuity of the solution map}

To prove the local
boundedness and the upper semicontinuity of the solution map, we need the following lemma.

\begin{lemma}\label{open_cone} The set $\Ri_{d}$ is open in $\R_{d}[x]$.
\end{lemma}
\begin{proof} To prove the openness of $\Ri_{d}$, we only need to show that the complement $\R_{d}[x]\setminus\Ri_{d}$ is closed. Clearly, $\R_{d}[x]\setminus\Ri_{d}=\R_{d-1}[x]\cup\U_{d}$. Let $\{g^k\}$ be a sequence in $\R_{d}[x]\setminus\Ri_{d}$ such that $g^k\to g$. From the definition of $\Ri_{d}$, if $\deg g<d$, i.e., $g\in\R_{d-1}[x]$, then $g\in\R_{d}[x]\setminus\Ri_{d}$. Thus, we can suppose that $\deg g=d$. One has $g^k_d\to g_d$, here $g^k_d$ is the component of degree $d$ of $g^k$.

We now prove that $g$ belongs to $\U_{d}$. For each $k$, $\Sol(K_{\infty},g^k_d)$ is unbounded. There exists an unbounded sequence  $\{x^k\}$ such that $x^k\in\Sol(K_{\infty},g^k_d)$,  $\|x^k\|\to+\infty$, $\|x^k\|^{-1}x^k\to\bar x$ with $\|\bar x\|=1$. Let $v\in K_{\infty}$ be given. One has $\|x^k\|v\in K_{\infty}$ and
	$g^k_d(\|x^k\|v)\geq g^k_d(x^k),$
	for any $k$.
	Dividing the last inequality by $\|x^{k}\|^{d}$ and letting $k\to+\infty$, one has
	$g^k_d(v)\geq g^k_d(\bar x)$. This conclusion holds for every $v\in K_{\infty}$.
	This yields $\bar x \in \Sol(K_{\infty},g_d)$.
	As $\|\bar x\|=1$, we have $\bar x\neq 0$. It follows that $g$ belongs to $\U_{d}$.  The closedness of $\R_{d}[x]\setminus\Ri_{d}$ is proved. \qed

\if We now prove that the complement $\R_{d}[x]\setminus\Oo_{d}$ is closed. Let
$\{g^k\} \subset \R_{d}[x]\setminus\Oo_{d}$ be a sequence with $g^k\to g$. From
the second equation in \eqref{POR}, we prove that $g$ belongs to $\R_{d-1}[x]\cup
\E_{d}\cup\U_{d}$. There is a subsequence of $\{g^{k_p}\}$ such that it is
contained in $\R_{d-1}[x]$, $\E_{d}$, or $\U_{d}$. If $\{g^{k_p}\}\subset
\R_{d-1}[x]$ then, by the closedness of $\R_{d-1}[x]$, one has $g^{k_p}\to g\in
\R_{d-1}[x]$. Now we consider the case that  $\{g^{k_p}\}\subset
\E_{d}\cup\U_{d}$, i.e., $g^{k_p}$ is non-negative over
$K_{\infty}$, for any $p$. For any $v\in K_{\infty}$, one has $g_d^{k_p}(v)\geq
0$. Let $p \to +\infty$, we have $g_d(v)\geq 0$. Hence, $g_d$ is
non-negative over $K_{\infty}$. This means that $g\in \E_{d}\cup\U_{d}$.  The
openness of $\Oo_{d}$ is proved. \fi

\end{proof}

Recall that a set-valued map $\Psi:\R^m\rightrightarrows\R^n$ is
\textit{locally bounded} at $\bar u$ if there exists an open neighborhood $U$
of $\bar u$ such that $\cup_{u\in U} \Psi(u)$ is bounded \cite[Definition
5.14]{RW}. The  map $\Psi$ is  \textit{upper semicontinuous} at $\bar u\in T$
if for
any open set $V\subset \R^n$ such that $\Psi(\bar u)\subset V$ there exists a
neighborhood $U$ of $\bar u$ such that $\Psi(u)\subset V$ for all $u\in U$.
Recall that if $\Psi$ is closed, namely, the graph
$$\gph(\Psi):=\big\{(u,v)\in \R^m\times \R^n: v\in \Psi(u)\big\}$$
is closed in $\R^m\times \R^n$, and locally bounded at $u$, then $\Psi$ is
upper semicontinuous at $u$ \cite[Theorem 5.19]{RW}.

\begin{proposition}\label{prop:local} Assume that $K$ is convex. If $\OP(K,f)$
is regular, then the following statements hold:
	\begin{description}
		\item[\rm(a)] The solution map $\Sol_K(\cdot)$ is locally bounded at $f$, i.e., there exists $\varepsilon>0$ such that the set
		\begin{equation}\label{O_eps}
		O_{\varepsilon}:=\bigcup_{g\in \B(\varepsilon,d)} \Sol(K,f+g),
		\end{equation}
		where $\B(\varepsilon,d)$ is the open ball in $\R_{d}[x]$ with center $0$
		and radius $\varepsilon$, is bounded.
		\item[\rm(b)] The solution map $\Sol_K(\cdot)$ is upper semicontinuous at $f$.
	\end{description}
\end{proposition}

\begin{proof} $\rm(a)$ According to Lemma \ref{open_cone}, $\Ri_{d}$ is open in $\R_{d}[x]$. There is a closed ball $\overline\B(\varepsilon,d)$  such that
	\begin{equation}\label{PB}
	f+\overline\B(\varepsilon,d)\subset \Ri_{d}.
	\end{equation}
Assume to the contrary that $O_{\varepsilon}$ is unbounded. Then, there exists an unbounded sequence $\{x^k\}$ and a sequence $\{g^k\}\subset\B(\varepsilon,d)$ such that $x^k$ solves $\OP(K,f+g^k)$ with $x^k\neq 0$ for every $k$, $\|x^k\|\to+\infty$, and
	$\|x^k\|^{-1}x^k\to\bar x$ with $\|\bar x\|=1$.
	By the compactness of $\overline\B(\varepsilon,d)$, without loss of generality, we can assume that $g^k\to g$ with $g\in \overline\B(\varepsilon,d)$.

From assumptions, for every $k$, one has
	\begin{equation}\label{VI_k}
	(f+g^k)(y)\geq (f+g^k)(x^k),
	\end{equation}
for any $y\in K$. Let $y\in K$ be fixed and assume that $v\in K_{\infty}$. By the
	convexity
	of $K$, one has $y+\|x^k\|v\in K$ for any $k$. From \eqref{VI_k}, we
	conclude that
	$$(f+g^k)(y+\|x^k\|v)\geq (f+g^k)(x^k).$$
	Dividing this inequality by $\|x^k\|^{d}$ and taking $k\to+\infty$, we obtain
	$$(f+g)_d(v)\geq (f+g)_d(\bar x).$$
The conclusion hold for any $v\in K_{\infty}$. It follows that $\bar x\in\Sol\left(K_{\infty},(f+g)_d\right)$. Because of
	\eqref{PB}, $\Sol\left(K_{\infty},(f+g)_d\right)$ is contained in $\{0\}$,
	which contradicts to $\|\bar x\|=1$. Hence, $O_{\varepsilon}$ must be
	bounded.

	$\rm(b)$ It is not difficult to prove that the graph
	$$\gph(\Sol):=\big\{(g,x)\in\R_{d}[x]\times \R^n: x\in \Sol(K,g)\big\}$$
	is closed in $\R_{d}[x]\times\R^n$. Since $\Sol_K(\cdot)$ is locally bounded on $\Ri_{d}$, according to \cite[Theorem 5.19]{RW}, $\Sol_K(\cdot)$ is upper semicontinuous at $f$.  \qed
\end{proof}

\subsection{Local upper-H\"{o}lder stability of the solution map}

When the constraint set $K$ is convex and given by polynomials, we can
investigate the local upper-H\"{o}lder stability of the solution map under the
regular condition. To prove the stability, we need the
following lemma.

\begin{lemma}[\cite{LMP14}]\label{lm:Holder} Let $U$ be a semi-algebraic subset
in $\R^n$, represented by
	\begin{equation*}
	U = \left\lbrace x \in \R^n : u_i(x) = 0,i\in[l], v_j(x)\leq 0,j\in[m]\right\rbrace,
	\end{equation*}
	where $u_i(x),i\in[l],$  and $v_j(x), j\in[m],$ are polynomials. For any
	compact set $V\subset\R^n$, there are constants $c>0$ and $H>0$ such that
	$$d(x,U)\leq c\Big(\sum_{i=1}^{l}|u_i(x)|+\sum_{j=1}^{m}[v_j(x)]_+ \Big)^H,
	$$
	for all $x\in V$, here $[r]_+:=\max\{0,r\}$ and $d(x,U)$ the usual distance
	from $x$ to the set $U$.
\end{lemma}

\begin{theorem}
Assume that $\OP(K,f)$ is regular and $K$ is a convex set given by
	\begin{equation*}\label{K}
	K = \left\lbrace x \in \R^n : p_i(x) = 0,i\in[l], q_j(x)\leq 0,j\in[m]\right\rbrace,
	\end{equation*}
where all $p_i,q_j$ are polynomials. If $\Sol(K,f)$ is nonempty, then the map $\Sol_K(\cdot)$ is locally upper-H\"{o}lder stable at $f$, i.e., there exist $\ell>0,H>0$ and $\varepsilon>0$ such that
	\begin{equation}\label{eq:Hol}
	\Sol(K,g)\subset \Sol(K,f)+\ell\|g-f\|^{H}\Bo,
	\end{equation}
	for all $g\in\R_{d}[x]$ satisfying $\|g-f\|<\varepsilon$, where $\Bo$ is the closed unit ball in $\R^n$.
\end{theorem}

\begin{proof} Suppose $\Sol(K,f)$ is nonempty and its optimal value is $f^*$.
Since $\OP(K,f)$ is regular and $K$ is convex, according to Proposition
\ref{prop:local}, there exists $\varepsilon>0$ such that $\Sol(K,f)\subset
O_{\varepsilon}$, defined by \eqref{O_eps}, is bounded. Let $V$ be the closure
of
$O_{\varepsilon}$. It follows that $V$ is a nonempty compact set. By the
assumptions, we see that
	\begin{equation*}\label{SolHol}
	\Sol(K,f)=\{x \in \R^n : f(x)-f^*=0, p_i(x) = 0,i\in[l], q_j(x)\leq 0,j\in[m] \}.
	\end{equation*}
From this equality, by applying Lemma \ref{lm:Holder} for $U=\Sol(K,f)$ and the
compact set $V$, there are constants $c_0>0$ and $H>0$ such that
	\begin{equation}\label{dxK}
	d(x,\Sol(K,f))\leq c_0A(x)^H \ \ \forall x\in V,
	\end{equation}
where
$$A(x):=|f(x)-f^*|+\sum_{i=1}^{l}|p_i(x)|+\sum_{j=1}^{m}[q_j(x)]_+.$$

Let $g\in\R_{d}[x]$ be arbitrary given such that $\|g-f\|<\varepsilon$. From the	definition of $V$, $\Sol(K,f)$ and $\Sol(K,g)$ are subsets of $V$. Here, $\Sol(K,g)$ may be empty. By the compactness of $V$, we define the constant $L:=\max\{\|X(x)\|:x\in V\}$. Hence, one has
	\begin{equation}\label{g_f}
	|g(x)-f(x)|\leq \|X(x)\|\|g-f\|\leq L\|g-f\| \ \ \forall x\in V.
	\end{equation}

If $\Sol(K,g)$ is empty, then \eqref{eq:Hol} is obvious. Thus, we consider the
case that $\Sol(K,g)\neq\emptyset$. Since both $\Sol(K,f)$ and $\Sol(K,g)$ are
nonempty and compact, for any $x_g\in\Sol(K,g)$, there is $x_f\in \Sol(K,f)$
such that
	\begin{equation}\label{y_z}
	\|x_g-x_f\|=d(x_g,\Sol(K,f)).
	\end{equation}
Because of $p_i(x_g) = 0$ for $i\in[l]$ and $q_j(x_g)\leq 0$ for $j\in[m]$,
from the definition of $A(x)$, one has $A(x_g)=|f(x_g)-f^*|$. By \eqref{y_z}
and \eqref{dxK}, we see that
	$$\|x_g-x_f\|  \leq c_0A(x_g)^H=c_0|f(x_g)-f^*|^H.$$
	Since $x_f\in \Sol(K,f)$, we have $f(x_f)=f^*\leq f(x_g)$. Therefore, we obtain
	\begin{equation}\label{y_z2}
	\|x_g-x_f\|\leq c_0|f(x_g)-f^*|^H=c_0(f(x_g)-f(x_f))^H.
	\end{equation}
It follows from $x_g\in\Sol(K,g)$ that $g(x_g)-g(x_f)\leq 0$. Since $x_g,x_f\in V$, we conclude from  \eqref{g_f} that
	$$\begin{array}{ll}
	f(x_g)-f(x_f)   & = \; (f(x_g)-g(x_g))+(g(x_g)-g(x_f))+(g(x_f)-f(x_f))  \smallskip \\
	&\leq  \; (f(x_g)-g(x_g))+(g(x_f)-f(x_f)) \smallskip \\
	&\leq \; 2L \|g-f\|.
	\end{array}$$
	The inequality \eqref{y_z2} and the last result lead to
	$$\|x_g-x_f\|\leq c_0(2L)^H\|g-f\|^H,$$
	consequently,
	$$d(x_g,\Sol(K,f))=\|x_g-x_f\|\leq \ell\|g-f\|^H,$$
	where $\ell=c_0(2L)^H$.

The conclusion holds for any $x_g$ in $\Sol(K,g)$. 	Hence, the inclusion \eqref{eq:Hol} of the theorem is proved.  \qed
\end{proof}

\section{Genericity of the regularity condition}\label{sec:gen}

In this section, we discuss the genericity of the regularity condition of polynomial optimization problems.

A subset $A$ is called \textit{generic} in  $\R^m$ if $A$ contains a countable
intersection of dense and open sets in $\R^m$. If $A$ is generic in  $\R^m$ and
$A\subset B$ then $B$ also is generic in  $\R^m$. Let $\mathcal{T}$ be a
topological space. It is known that if $h: \R^m\to \mathcal{T}$ is a
homeomorphism and $A$ is generic in $\R^m$, then $h(A)$ is generic in
$\mathcal{T}$.

Let $U\subset \R^m$ be a semi-algebraic set. Then, there exists a decomposition
of $U$ into a disjoint union \cite[Theorem 2.3.6]{BCF98}, $U=\cup_{i=1}^sU_i$,
where each $U_i$ is semi-algebraically homeomorphic to $(0,1)^{d_i}$.
Here, let $(0,1)^{0}$ be a point,  $(0,1)^{k}\subset \R^{k}$ be the set of
points $x=(x_1,\dots,x_{k})$ such that $x_j\in (0,1)$ for all $j\in[k]$.
The \textit{dimension} of $U$ is defined by $\dim(U):=\max\{d_1,\dots,d_s\}.$
The dimension is well-defined and does not depends on the decomposition of $U$.
Recall that if $U$ is nonempty and $\dim(U)$ is zero, then
$U$ has finitely many points. Furthermore, if $\dim(\R^m\setminus U)<m$, then
$U$ is generic in $\R^m$ (see, e.g. \cite[Lemma 2.3]{DHP16}).

The space generated by all monomials of degree $d$ listed by lexicographic
ordering $\{x_1^d,x_1^{d-1}x_2,x_1^{d-1}x_3,\dots,x_n^d\}$ is denoted by
$\Hd_d$. One has the direct sum $\R_{d}[x]=\Hd_d \oplus \R_{d-1}[x]$. The dimension
of $\Hd_d$ is denoted by $\eta$. 	For every homogeneous polynomial $h\in
\Hd_d$, one has a unique $b\in \R^{\eta}$,
such that $h(x)=b^TX_d(x)$, where
$$X_d^T(x)=(x_1^d,x_1^{d-1}x_2,x_1^{d-1}x_3,\dots,x_n^d).$$
Here, $\nabla (b^TX_d(x))$ is the gradient vector of $b^TX_d(x)$ and
$D_{b}[\nabla (b^TX_d(x))]$ is the Jacobian matrix of $b^TX_d(x)$ with respect
to $b$.

\begin{lemma}\label{rank}  One has
	$\rank(D_{b}[\nabla (b^TX_d(x))])=n$ for all $x\in \R^n\setminus\{0\}$.
\end{lemma}

\begin{proof} In the proof, we are only interested in the monomials $x_i^{d-1}x_j$, where $i,j\in[n]$. Hence, for convenience, we rewrite $X_d^T(x)$ and $b^T$ respectively as follows:
$$(x_1^d,x_1^{d-1}x_2,\dots,x_1^{d-1}x_n; \ \dots; \
x_n^{d-1}x_1,x_n^{d-1}x_2,\dots,x_n^{d}; \ \dots)$$
and
$$(b_{11},b_{12},\dots,b_{1n}; \ \dots; \
b_{n1},b_{n2},\dots,b_{nn}; \ \dots).$$
	Then, we have
	\begin{equation}\label{h}
	b^TX_d(x)=\sum_{j\in [n]}b_{1j}x_1^{d-1}x_j+\dots+\sum_{j\in
	[n]}b_{nj}x_n^{d-1}x_j+Q,
	\end{equation}
	where $Q$ is a homogeneous polynomial of degree $d$.

From \eqref{h}, an easy computation shows that
	$$\frac{\partial (b^TX_d(x))}{\partial x_i}=db_{ii}x_i^{d-1}+(d-1)\sum_{j\neq i}b_{ij}x_i^{d-2}x_j +\sum_{j\neq i}b_{ji}x_j^{d-1}+\frac{\partial Q}{\partial x_i},$$
	and the $n\times \eta$-matrix  $D_{b}[\nabla (b^TX_d(x))]$ can be described as follows
	$$D_{b}[\nabla (b^TX_d(x))]=\Big[M_1,M_2,\cdots,M_n,\cdots\ \Big],$$
	where the submatrix $M_i$, for $i\in[n]$, is defined by
	$$M_i=\begin{bmatrix}
	x_i^{d-1}I_{i-1}&O_{i\times 1}&O_{i\times(n-i)} \smallskip \\
	L_{1\times (i-1)}& dx_i^{d-1}& R_{1\times(n-i)} \smallskip \\
	O_{i\times(n-i)} & \ O_{{(n-i)}\times 1} \ &  x_i^{d-1}I_{n-i}\\
	\end{bmatrix}$$
with $I_{k}$ being the unit $k\times k$-matrix, $O_{k\times s}$ being the zero
$k\times s$-matrix,
	$$L_{1\times (i-1)}=\Big((d-1)x_i^{d-2}x_1,\dots,(d-1)x_i^{d-2}x_{i-1}\Big),$$
	and
	$$R_{1\times (n-i)}=\Big((d-1)x_i^{d-2}x_{i+1},\dots,(d-1)x_i^{d-2}x_{n}\Big).$$
	We observe that $\det(M_i)=dx_i^{d(d-1)}$, for all $i\in[n]$. Since $x\neq 0$, there exists $l\in[n]$ such that $x_l\neq 0$. This implies that $\rank(M_l)=n$. Hence, the rank of $D_{b}[\nabla (b^TX_d(x))]$ is $n$, for any $x\neq 0$. \qed
\end{proof}

Suppose that $C$ is a polyhedral cone given by
\begin{equation}\label{K_0}
C = \left\lbrace x \in \R^n : Ax \geq 0\right\rbrace,
\end{equation}
where  $A=(a_{ij})\in {\mathbb R}^{p\times n}$. Let $\KKT(C,g)$, where
$g\in\R_{d}[x]$, be the set of the Karush-Kuhn-Tucker points of $\OP(C,g)$, i.e.,
$x\in \KKT(C,g)$ if and only if
there exists $\lambda\in \R^p$ such that
\begin{equation}\label{KKT_p}
\left\lbrace
\begin{array}{l}\nabla g(x)-A^T\lambda =0, \\
\lambda^T(Ax)=0, \; \lambda \geq 0 , \; Ax \geq 0. 	\end{array}\right.
\end{equation}
From the Karush-Kuhn-Tucker conditions, we see that $\Sol(C,g)\subset \KKT(C,g)$ for all $g\in\R_{d}[x]$.

For each index set $\alpha \subset [p]$, we associate the \textit{pseudo-face} $C_{\alpha}$ of $C$, which is denoted and defined by
$$C_\alpha:=\Big\{x\in {\mathbb R}^n\,:\, \sum_{j=1}^{n}a_{ij}x_j=0\ \,
\forall i\in\alpha,\ \, \sum_{j=1}^{n}a_{ij}x_j> 0\ \,\forall i\in[p]\setminus\alpha\Big\},$$
where $a_{ij}$ is the element in the $i$-th row and the $j$-th column of $A$. The number of pseudo-faces of $C$ is finite. These pseudo-faces  establish a disjoint decomposition of $C$. So, we obtain
\begin{equation}\label{Sol_decom}
\KKT(C,g)=\displaystyle\bigcup_{\alpha\subset [p]}\left(\KKT(C,g)\cap C_{\alpha}\right) ,
\end{equation}

The following proposition shows that the Karush-Kuhn-Tucker set-valued map of homogeneous polynomial optimization problems
$$\KKT_C:\R^{\eta}\rightrightarrows\R^n, \ b\mapsto \KKT_C(b)=
\KKT(C,b^TX_d(x)),$$
is finite-valued, i.e., the cardinal $\#\KKT_C(b)$ is finite, on a generic semi-algebraic set of $\R^{\eta}$.

\begin{proposition}\label{generic_1} Assume that $C$ is a polyhedral cone given
by
\eqref{K_0} and the matrix $A$ is full rank. Then, there exists a  generic
semi-algebraic set $\Sa\subset\R^{\eta}$ such that $\#\KKT_C(b)<\infty$ for any
$b\in\Sa$.
\end{proposition}
\begin{proof} Let $C_{\alpha}$ be a nonempty pseudo-face of $C$ and $0\notin C_{\alpha}$. This implies that
	$X_d(x)$ is nonzero on this pseudo-face.	We consider the function
	$$\Phi_{\alpha}:\R^{\eta}\times C_{\alpha}\times \R_+^{|\alpha|} \to \R^{n+|\alpha|},$$
	which is defined by
	$$\Phi_\alpha(b,x,\lambda_\alpha)=\Big(\nabla(b^TX_d(x))+\displaystyle\sum_{i\in \alpha}\lambda_i A_i, A_\alpha x\Big),$$
	where
	$A_\alpha x=(A_{i_1}x,\dots,A_{i_{|\alpha|}x}), i_j\in\alpha.$ Clearly, $C_{\alpha}$ is smooth and $\Phi_\alpha$ is a
	semi-algebraic function of class $C^{\infty}$.
	The Jacobian matrix
	of $\Phi_\alpha$ is determined as follows
	$$D\Phi_\alpha=\left[ \begin{array}{c|c|c}
	D_{b}[\nabla (b^TX_d(x))] \ \ & \  * \ \ &\ \  A_\alpha^T  \ \  \\  \hline
	O_{|\alpha|\times \eta}&\ \ \ \ A_\alpha \ \ \ \ & \ O_{|\alpha|\times |\alpha|}
	\end{array}\right].$$
From Lemma \ref{rank}, for all $x\in C_\alpha$, the rank of $D_{b}[\nabla
(b^TX_d(x))]$ is $n$. Since the rank of $A_\alpha$ is $|\alpha|$, we conclude
that the rank of the matrix $D\Phi_\alpha$ is $n+|\alpha|$ for all $x\in
C_\alpha$. Therefore,	$0\in \R^{n+|\alpha|+|J|}$ is a regular value of
$\Phi_\alpha$. According to the Sard Theorem with parameter \cite[Theorem
2.4]{DHP16}, there exists a generic semi-algebraic set  $\Sa_{\alpha}\subset
\R^{\eta}$ such that if $b\in \Sa_{\alpha}$ then $0$ is a regular value of the
map
	$$\Phi_{\alpha,b}:C_{\alpha}\times \R^{|\alpha|} \to \R^{n+|\alpha|}, \ \Phi_{\alpha,b}(x,\lambda_\alpha) =\Phi_\alpha(b,x,\lambda_\alpha).$$
We see that $\Omega(\alpha,b):=\Phi^{-1}_{\alpha,b}(0)$ is a semi-algebraic
set. From the Regular Level Set Theorem \cite[Theorem 9.9]{Tu2010}, we can
claim
that if $\Omega(\alpha,b)$ is nonempty then the set is a $0-$dimensional
semi-algebraic set. It
follows that $\Omega(\alpha,b)$ is a finite set. Moreover, from \eqref{KKT_p},
one has
	$\KKT_C(b)\cap C_{\alpha}=\pi(\Omega(\alpha,b)),$
	where $\pi$ is the projection $\R^{n+|\alpha|} \to \R^n$ which is defined by $\pi(x,\lambda_{\alpha}) = x$. Hence, $\KKT_C(b)\cap C_{\alpha}$ is a finite set.

We consider the case that $0\in C_{\alpha}$ and define
$U:=C_{\alpha}\setminus\{0\}$. As is clear, $U$ is semi-algebraic. From
\eqref{KKT_p}, we see that $0\in \KKT_C(b)$. Hence,
	$$\KKT_C(b)\cap C_{\alpha}=\{0\}\cup(\KKT_C(b)\cap U).$$
From the previous argument, $\KKT_C(b)\cap U$ is a finite set. By the decomposition \eqref{Sol_decom}, 	$\KKT_C(b)$ is a finite set.

	Take $\Sa=\cap_{\alpha\subset [p]}\Sa_{\alpha},$
	it follows that $\Sa$ is generic in $\R^{\eta}$ and	$\KKT_C(b)$ has finite points for any $b\in\Sa$. Hence, $\#\KKT_C(b)<\infty$ for all $b$ in $\Sa$. The proof is complete. \qed
\end{proof}

\begin{corollary}\label{cor:gen_1} Assume that $C$ is a polyhedral cone given
by
\eqref{K_0} and the matrix $A$ is full rank. Then there exists a generic set
$\mathcal{G}_d$ in $\Hd_d$ such that $\#\Sol(C,g)<\infty$ for any
$g\in\mathcal{G}_d$.
\end{corollary}

\begin{proof}
	Since $\R^{\eta}$ and $\Hd_d$ are homeomorphic, with the isomorphism $\Pi: \R^{\eta} \to \Hd_d$ defined by $\Pi(b)=b^TX_d(x)$. According to Proposition \ref{generic_1}, there exists a generic set $\Sa\subset \R^{\eta}$ such that the Karush-Kuhn-Tucker set $\KKT(C,b)$ is finite, for any $b\in\Sa$. Clearly, $\mathcal{G}_d:=\Pi(\Sa)$ is generic in $\Hd_d$. Since $\Sol(C,b^TX_d(x))\subset \KKT(C,b)$,  one has $\#\Sol(C,g)<\infty$, for any $g\in\mathcal{G}_d$.  \qed
\end{proof}

\begin{remark}\label{Bel_Kla} If the constraint $K$ is represented by
	\begin{equation}\label{K_1}
	K=\left\lbrace x \in \R^n : q_1(x)\leq 0,\dots,q_m(x)\leq 0\right\rbrace,
	\end{equation}
	where $q_1,\dots,q_m$ are convex polynomials, then the recession cone of $K$ is a nonempty polyhedral cone. We denote
	$$K^j=\left\lbrace x \in \R^n : q_j(x)\leq 0\right\rbrace, \ j\in [m].$$
	For each $j\in [m]$, $K^j$ is closed convex set, and $K_{\infty}^j$ is polyhedral (see \cite[p.39]{BeKla2002}). Since $K=K^1\cap \dots\cap K^m$, according to \cite[Proposition 2.1.9]{AuTe2003}, one has
	$$K_{\infty}=\bigcap_{j\in[m]}K_{\infty}^j.$$
	If follows that $K_{\infty}$ is a nonempty polyhedral cone. Hence, there
	exists a matrix $A\in\R^{p\times n}$ such that \begin{equation}\label{KAx}
	K_{\infty}=\{x\in\R^n: Ax\geq 0\}.
	\end{equation}
\end{remark}

\begin{theorem}
Assume that $K$ be represented by \eqref{K_1} and the cone $K_{\infty}$ represented by \eqref{KAx}, where $A$ is full rank. Then, the set $\Ri_{d}$ is  generic in $\R_{d}[x]$.
\end{theorem}
\begin{proof} From Remark \ref{Bel_Kla}, the recession cone $K_{\infty}$ is a nonempty polyhedral cone, where $K_{\infty}=\{x\in\R^n: Ax\geq 0\}$. According to Corollary \ref{cor:gen_1}, there exists a generic set $\mathcal{G}_d$ in $\Hd_d$ such that $\#\Sol(K_{\infty},g)<\infty$ for any $g\in\mathcal{G}_d$. Because of the direct sum $\R_{d}[x]=\Hd_d\oplus\R_{d-1}[x]$, the set $\mathcal{G}_d\oplus\R_{d-1}[x]$ is  generic in $\R_{d}[x]$. It is easy to check that $\mathcal{G}_d\oplus\R_{d-1}[x]\subset \Ri_{d}$. Hence, $\Ri_{d}$ is generic in $\R_{d}[x]$. \qed
\end{proof}

\begin{example} Consider the problem $\OP(K,f)$ given in Example \ref{ex1}, we see that $$\Ri_{2}=\{a_2x^2+a_1x+a_0:a_2\neq 0,a_1\in\R,a_0\in\R\}$$
	is open and dense in $\R_{2}[x]$.
\end{example}

\section*{Perspectives}
The regularity condition enables us to investigate the stability of the optimal
value function of polynomial optimization problems. Furthermore, the regularity
 condition is useful to study the connectedness of the solution sets of convex
 polynomial vector optimization
problems.

\begin{acknowledgements}
	The author would like to thank the anonymous referees for their corrections and comments. This work has been supported by European Union's Horizon 2020 research and innovation programme under the Marie Sk\l{}odowska-Curie Actions, grant agreement 813211 (POEMA).
\end{acknowledgements}

\end{document}